\documentclass[onefignum,onetabnum]{siamonline250211}

\usepackage[T1]{fontenc}
\usepackage{amsmath, amssymb, mathtools}
\usepackage{amsfonts}
\usepackage{graphicx}
\usepackage{epstopdf}
\usepackage{algorithmic}
\ifpdf
  \DeclareGraphicsExtensions{.eps,.pdf,.png,.jpg}
\else
  \DeclareGraphicsExtensions{.eps}
\fi

\usepackage{enumitem}
\setlist[enumerate]{leftmargin=.5in}
\setlist[itemize]{leftmargin=.5in}

\newsiamremark{remark}{Remark}
\newsiamremark{hypothesis}{Hypothesis}
\crefname{hypothesis}{Hypothesis}{Hypotheses}
\newsiamthm{claim}{Claim}
\newsiamremark{fact}{Fact}
\crefname{fact}{Fact}{Facts}

\newcommand{\R}{\mathbb{R}}

\newcommand{\Sph}[1]{\mathbb{S}^{#1}}

\newcommand{\E}{\mathbb{E}}
\newcommand{\FvML}{\mathrm{FvML}}
\newcommand{\dd}{\mathop{}\!\mathrm{d}}
\newcommand{\sigmaN}{\sigma^{\otimes N}}
\newcommand{\samples}{(\Sph{n-1})^{N}}

\DeclarePairedDelimiter{\norm}{\lVert}{\rVert}
\DeclarePairedDelimiter{\abs}{\lvert}{\rvert}

\headers{Unbiased estimation of squared concentration}{Z. Jabbar, Y. Jiang, and A. A. Popov}

\title{Unbiased estimation of squared concentration in the Fisher-von~Mises–Langevin distribution and the impossibility of unbiased concentration\thanks{Submitted to the editors May 2026.
\funding{This work was funded by startup funds from the University of Hawaii.}}}

\author{Zain Jabbar\thanks{Department of Mathematics, University of Hawaii at Manoa, Honolulu, HI 
  (\email{zjabbar@hawaii.edu}).}
\and Yuqin Jiang\thanks{Department of Geography and Environment, University of Hawaii at Manoa, Honolulu, HI 
  (\email{yuqinj@hawaii.edu}).}
\and Andrey A. Popov\thanks{Department of Information and Computer Sciences, University of Hawaii at Manoa, Honolulu, HI 
  (\email{apopov@hawaii.edu}).}}

\usepackage{amsopn}

\usepackage[dvipsnames]{xcolor}
\usepackage{pgfplots}
\pgfplotsset{compat=1.18}

\pgfplotsset{clean/.style={axis lines*=left,
        axis on top=true,
        axis x line shift=0.0em,
        axis y line shift=0.75em,
        every tick/.style={black, thick},
        axis line style = ultra thick,
        tick align=outside,
        clip=false,
        x=0.05cm,
        y=1.7cm,
        major tick length=4pt}}

\definecolor{tolblue}{rgb}{0.2667,0.4667,0.6667}
\definecolor{tolred}{rgb}{0.9333,0.4000,0.4667}
\definecolor{tolyellow}{rgb}{0.8000,0.7333,0.2667}
\definecolor{tolcyan}{rgb}{0.4000,0.8000,0.9333}
\definecolor{tolpurple}{rgb}{0.6667,0.2000,0.4667}
\definecolor{tolgreen}{rgb}{0.1333,0.5333,0.2000}

\usepackage[listings,skins,breakable]{tcolorbox}
{%
\end{tcolorbox}\endgroup}

{%
\end{tcolorbox}\endgroup}

{%
\end{tcolorbox}\endgroup}

\ifpdf
\hypersetup{
  pdftitle={Unbiased Partial-Sum Estimation for the Squared FvML Concentration},
  pdfauthor={Zain Jabbar, Yuqin Jiang, and Andrey A. Popov}
}
\fi

\begin{document}

\maketitle

\begin{abstract}
    The estimation of concentration parameter in Fisher-von~Mises Langevin distribution is the directional statistics analogue of the estimation of the precision matrix for the Gaussian distribution. 
    In this work we show that unbiased estimation of this parameter is impossible.
    With this realization in hand, we provide an alternative parameterization of the Fisher-von~Mises Langevin distribution in terms of the squared concentration, which we term the intensity.
    We fruther show that unbiased estimation of thereof is possible, and provide (almost) unbiased estimators thereof in terms of a partial sum U-statistic.
    We showcase our new estimator on synthetic data, New York taxi trip data, and on spherical word embeddings.
\end{abstract}

\begin{abstract}
% We study which functions of the concentration parameter $\kappa$ of
% the Fisher--von Mises--Langevin (FvML) distribution on the sphere
% $\mathbb{S}^{n-1}$ admit unbiased estimation.
% We prove that $\kappa$ itself does not, for any $n\ge1$ and $N\ge1$,
% and obtain a necessary condition: every unbiasedly estimable function
% of $\kappa$ must be real-analytic in $\kappa^2$.

% For $n\ge2$, $M\ge1$, and $N\ge2M$, we construct a U-statistic
% $\hat\zeta_M$ that is exactly unbiased for the finite partial sum
% $\zeta_M(\kappa)=\sum_{\ell=1}^M c_\ell A_n(\kappa)^{2\ell}$.
% This is the $M$-term truncation of the convergent inverse Bessel-ratio
% series $\kappa^2=\sum_{\ell\ge1}c_\ell A_n(\kappa)^{2\ell}$ on its
% convergence interval.
% Empirical examples separate truncation error (synthetic),
% bootstrap stability (NYC taxi flows \cite{taxi2020taxi}), and
% high-dimensional descriptive use (JoSE word embeddings
% \cite{meng2019spherical}).
\end{abstract}



\section{Introduction}
\label{sec:introduction}

Directional data arise in a wide range of scientific and applied domains, including classical examples such as geological fault plane analysis, 
\cite{fisher1953dispersion}, 
urban mobility~\cite{jiang2024entropy}, 
arrival times~\cite[\S1.1]{mardia2009directional},
attitude estimation~\cite{shuster1993survey},  word embeddings~\cite{meng2019spherical}, and is important for filtering applications~\cite{kurz2019directional, kurz2016unscented}.
In each of these cases the dominant statistical question is the same: what is the mean tendency of the directions, and how widely do the data spread therefrom.
These data can be represented on the hypersphere $\Sph{n-1}\subset\R^n$, with $n$ representing the underlying dimension in which the data reside, from $n=2$ in the case of planar directions, to $n=4$ for quaternions, to $n=100$ or more for machine learning applications.
 % \footnote{Code, data caches, and figure/table-regeneration scripts: \url{https://github.com/Zaijab/fvml_concentration_estimation}.}

The Fisher–von Mises–Langevin (FvML) distributions provides the canonical parametric model for such data. 
The FvML is parametrized by a mean direction $\mu\in\Sph{n-1}$, which can be viewed as a unit vector describing both the expected and most likely tendency of the data, and by a concentration parameter $\kappa>0$ which describes the spread of the data around this mean.
As $\kappa\to0$ the model approaches the uniform distribution on $\Sph{n-1}$, meaning that the data do not follow the mean direction at all.
Conversely as $\kappa$ increases, mass concentrates
around $\mu$, and the data tends to only point in same direction as the mean.
The concentration parameter $\kappa$ is the sole parameter controlling spread, playing a role analogous to (and derived from) precision in the multivariate Gaussian distribution.
we study which functions of $\kappa$ admit unbiased estimators.
Accurate inference on this concentration parameter $\kappa$ is therefore critical to inference tasks such as comparing directional patterns across groups, quantifying uncertainty in directional estimates, and calibrating directional models.
This work is focused on the following question: do unbiased estimators for $\kappa$ exist, and if not, what alternative parameterization of the FvML admits unbiased estimators?

Classical estimators target the concentration parameter $\kappa$ directly.  
The canonical method for which is the maximum-likelihood estimator (MLE)~\cite{fisher1953dispersion}, which is a consistent estimator.
For $n=2,3$ a closed-form expression for its finite-sample bias is
known \cite{bestfisher1981}, and for general $n$ the leading $O(N^{-1})$ bias term is known to be positive~\cite{tanabe2007parameter}.
Whether this bias is intrinsic to the inferential problem or an
artifact of the MLE has remained open.
This work settles this question: no unbiased estimator of $\kappa$ exists, for any dimension $n\ge1$ and any sample size $N\ge1$.
We prove that every unbiasedly estimable function of $\kappa$ must be real-analytic in $\kappa^2$, meaning that this obstacle is not specific to the MLE: \textbf{direct unbiased estimation of the concentration $\kappa$ is impossible}.

Armed with this realization, this work aims to provide unbiased estimators for a different parameterization of the FvML in terms of what we dub the intensity, which we take to be squared concentration. 
We construct a U-statistic that converges to an unbiased estimator of the intensity provided that other ancillary statistics can be computed from the data, and provide a numerical implementation thereof.

This work is organized as follows: \cref{sec:background} provides some background on the FvML, related Bessel-ratio
identities, sufficient statistic, and related estimators.
Next, \cref{sec:impossibility} proves the impossibility of unbiased estimation of $\kappa$.
\Cref{sec:estimators} constructs our partial sum U-statistic estimators and expresses the resulting truncation error.  
The next three sections represent empirical experiments then separate: synthetic data in~\cref{sec:synthetic}, NYC taxi flows in 
\cref{sec:taxi}, and word embeddings in \cref{sec:jose}.
Finally we end with some conclusions in~\cref{sec:conclusion}.

\section{Background}\label{sec:background}

This section provides background on the Fisher-von~Mises-Langevin distribution, sufficient statistics thereof, and the maximum likelihood estimator and $U$-statistics.

Throughout this section we assume that the manifold of interest is the sphere $\mathbb{S}^{n-1}$ embedded in $n$-dimensional Euclidean space $\mathbb{R}^n$. 
We additionally assume that $n \geq 2$, meaning that the we are either dealing with the circle or some higher-dimensional sphere.
The degenerate case of $n=1$ collapses to the sphere $\Sph{0}$ which is precisely the set $\{-1,+1\}$, and collapses the FvML family of distributions to a two-point exponential family, and is only briefly discussed.

\subsection{The Fisher-von~Mises-Langevin distribution}
\label{sec:fvml}

We now define the Fisher-von~Mises-Langevin distribution and properties thereof.

Let $\Sph{n-1}=\{x\in\R^n:\norm{x}=1\}$ and $\sigma$ the surface
measure on $\Sph{n-1}$ (the $(n-1)$-dimensional Hausdorff measure).
The \emph{Fisher-von~Mises-Langevin} (FvML) distribution
$\FvML(\mu,\kappa)$ with mean direction $\mu\in\Sph{n-1}$ and
concentration parameter $\kappa>0$ has $\sigma$-density,
\begin{equation}
  \label{eq:fvml-density}
  f(x;\mu,\kappa) = C_n(\kappa)\,e^{\kappa\mu^\top x},
  \quad
  C_n(\kappa) = \frac{\kappa^{n/2-1}}{(2\pi)^{n/2}\,I_{n/2-1}(\kappa)},
  \quad \kappa>0,
\end{equation}
where $I_\nu$ is the modified Bessel function of the first kind~\cite{NIST:DLMF}, defined by,
\begin{equation}
    I_\nu(\kappa)=\sum_{m=0}^\infty \frac{(\kappa/2)^{\nu+2m}}{m!\,\Gamma(\nu+m+1)}.
\end{equation}
The mean direction $\mu$ of the distribution defines both the mode of the distribution, and its Fr\'echet mean, meaning that data distributed as FvML are concentrated around this direction. 
The scalar concentration parameter $\kappa$ defines how concentrated the data are around the mean direction $\mu$.
A concentration parameter $\kappa$ of zero means that the data are essentially uniformly distributed around the sphere and are not at all concentrated around the mean direction $\mu$.
Conversely if the concentration parameter tends towards infinity, that means that the data are pointed solely in the direction of the mean direction $\mu$.
The estimation and analysis of the concentration $\kappa$ is the focus of this work.

The normalizing constant $C_n(\kappa)$ is dependent on the concentration $\kappa$ but is independent of the mean direction $\mu$ because $\sigma$ is invariant under the orthogonal group $O(n)$ \cite[p.~36]{mardia2009directional}.
The law extends continuously to $\kappa=0$, with,
\begin{equation}
    C^{-1}_n(0)=\sigma(\Sph{n-1}) = \frac{2\pi^{n/2}}{\Gamma(n/2)},
\end{equation}
recovering the uniform distribution.  
In this degenerate case the density is independent of $\mu$, so the mean direction is identifiable only for $\kappa>0$.
For $n=2$ this is known as the von~Mises distribution on the circle $\Sph{1}$ and for $n\ge3$ it is known as the von~Mises-Fisher distribution.

\subsection{The function \texorpdfstring{$A_n$}{A\_n}}
\label{sec:An}

It is known~\cite{fisher1953dispersion} that if $X$ is FvML distributed, then,
\begin{equation}\label{eq:An-to-EnormX}
    A_n(\kappa) = \mathbb{E}\left[\left\lVert X\right\rVert\right],
\end{equation}
where the function $A_n$ is the ratio of modified Bessel functions,
\begin{equation}
  \label{eq:An}
  A_n(\kappa) \coloneqq  \frac{I_{n/2}(\kappa)}{I_{n/2-1}(\kappa)},
  \qquad \kappa>0,
\end{equation}
thus the function $A_n$ has been studied for its properties in estimating the concentration parameter $\kappa$.

For any $\nu\ge0$, $\kappa^{-\nu}I_\nu(\kappa)$ extends to an entire function of $\kappa$ \cite[\S10.25(ii)]{NIST:DLMF}.
Explicitly writing out,
\begin{equation}
    A_n(\kappa)=\kappa\,\frac{\kappa^{-n/2}I_{n/2}(\kappa)}{
\kappa^{-(n/2-1)}I_{n/2-1}(\kappa)},
\end{equation}
exhibits the denominator as an entire function whose power series in $\kappa^2$ has positive coefficients $1/(m!\,2^{n/2-1+2m}\Gamma(n/2+m))$, hence no real zeros.
The function $A_n$ therefore extends to a real-analytic function on $\R$ with $A_n(0)=0$.
We now present some useful properties of $A_n$.

\begin{proposition}[Properties of $A_n$]
\label{prop:An-properties}
For $n\ge2$ the function $A_n$ satisfies:
\begin{enumerate}[label=(\roman*)]
  \item $A_n(\kappa) = \mathbb{E}_{\mu,\kappa}[\mu^\top X]$,
  the \emph{population mean resultant length}.  By the
  $O(n)$-equivariance of the FvML family,
  $\mathbb{E}_{\mu,\kappa}[X]=A_n(\kappa)\mu$.
  \item $A_n\colon(0,\infty)\to(0,1)$ is strictly increasing, with
  $A_n(\kappa)\to0$ as $\kappa\to0^+$ and $A_n(\kappa)\to1$ as
  $\kappa\to\infty$.
  \item The power series of $A_n$ has only odd powers:
  \begin{equation}
    \label{eq:An-series}
    A_n(\kappa) = \sum_{j=0}^\infty a_j\,\kappa^{2j+1},
    \quad
    a_0 = \frac{1}{n},
    \quad
    a_j = \frac{-1}{n+2j}\sum_{k=0}^{j-1}a_k\,a_{j-1-k}
    \quad j\ge1,
  \end{equation}
  converging for $|\kappa|<R_n$.  Since consecutive Bessel functions
  share no nonzero zeros, $R_n$ equals the modulus of the nearest
  complex zero of $I_{n/2-1}$.
  \item The coefficients alternate in sign: $(-1)^j a_j>0$ for all
  $j\ge0$.
\end{enumerate}
\end{proposition}

\begin{proof}
Property~(i): differentiating $\log C_n(\kappa)^{-1}$ gives
$A_n(\kappa)=\mathbb{E}_{\mu,\kappa}[\mu^\top X]$; see
\cite[p.~86]{mardia2009directional}.  The identity
$\mathbb{E}[X]=A_n(\kappa)\mu$ follows from the $O(n)$-equivariance of
the family.

Property~(ii): differentiating the log-partition function
$\log C_n(\kappa)^{-1}$ twice gives
$A_n'(\kappa)=\mathrm{Var}_{\mu,\kappa}(\mu^\top X)$; this variance
is strictly positive because $\mu^\top X$ is non-degenerate for
$n\ge2$.  So $A_n$ is strictly increasing on $(0,\infty)$.
The limit $A_n(\kappa)\to0$ as $\kappa\to0^+$ follows from
$A_n(0)=0$; the limit $A_n(\kappa)\to1$ as $\kappa\to\infty$
follows from the Bessel asymptotic
$I_\nu(\kappa)\sim e^\kappa/\sqrt{2\pi\kappa}$
\cite[\href{https://dlmf.nist.gov/10.30.E1}{Eq.~10.30.1}]{NIST:DLMF};
and $A_n(\kappa)<1$ because
$\|{\mathbb{E}[X]}\|=A_n(\kappa)<\mathbb{E}[\|X\|]=1$ by strict Jensen.

Property~(iii): the Bessel recurrence
$I_\nu'(\kappa)=I_{\nu-1}(\kappa)-\frac{\nu}{\kappa}I_\nu(\kappa)$
\cite[\href{https://dlmf.nist.gov/10.29.E1}{Eq.~10.29.1}]{NIST:DLMF}
yields the ODE (a Riccati equation, i.e.\ first-order and quadratic
in $A_n$)
\begin{equation}
  \label{eq:riccati}
  \kappa A_n' = \kappa - (n-1)A_n - \kappa A_n^2.
\end{equation}
Substituting $A_n(\kappa)=\sum_{j\ge0}a_j\kappa^{2j+1}$ and matching
coefficients of $\kappa^{2j+1}$ gives the recursion
in~\eqref{eq:An-series}.

Property~(iv):  We prove $(-1)^j a_j>0$ for all $j\ge0$
by induction. 
Define,
\begin{equation}
    S_j\coloneqq \sum_{k=0}^{j-1}a_k\,a_{j-1-k},
\end{equation}
so that $a_j=-S_j/(n+2j)$. 
In the base case: $a_0=1/n>0$.
For ($j\ge1$): assume $(-1)^k a_k>0$ for $0\le k\le j-1$.
Each term $a_k\,a_{j-1-k}$ has sign $(-1)^k\cdot(-1)^{j-1-k}=(-1)^{j-1}$, so every summand in $S_j$
carries the same sign.  Hence $\operatorname{sign}(S_j)=(-1)^{j-1}$ and $\operatorname{sign}(a_j)=-(-1)^{j-1}=(-1)^j$.
\end{proof}
The first several $a_j$ are recorded in \cref{tab:as}.

\begin{table}
  \centering
  \small
  \setlength{\tabcolsep}{4pt}
  \begin{tabular}{c||llll}
    \hline
    $j$ & 0 & 1 & 2 \\\hline
    $a_j$ & $\dfrac{1}{n}$ & $\dfrac{-1}{n^2(n+2)}$ & $\dfrac{2}{n^3(n+2)(n+4)}$\\
    \hline
  \end{tabular}
  \caption{First coefficients $a_j$ in the series of $A_n(\kappa)$.}
  \label{tab:as}
\end{table}

\subsection{Sufficient statistic and exponential family structure}
\label{sec:exp-family}

The FvML family is a natural exponential family with natural
parameter $\eta=\kappa\mu\in\R^n\setminus\{0\}$; the map
$(\mu,\kappa)\mapsto\eta$ is a bijection of
$\Sph{n-1}\times(0,\infty)$ onto $\R^n\setminus\{0\}$ with its inverse defined by $\kappa=\norm{\eta}$, and $\mu=\eta/\norm{\eta}$.

%Let $\sigmaN\coloneqq \sigma^{\otimes N}$ on $\samples$.

We now look at the joint distribution of a collection of $N$ samples from the FvML distribution.
The joint $\sigmaN$-density of $N$ i.i.d.\ observations
$X_1,\ldots,X_N\sim\FvML(\mu,\kappa)$ is defined as,
\begin{equation}
  \label{eq:joint-density}
  \frac{\dd P_{\mu,\kappa}^{\otimes N}}{\dd\sigmaN}(x)
  = C_n(\kappa)^N \exp\!\bigl(\kappa\mu^\top T\bigr),
  \qquad T \coloneqq  \sum_{i=1}^N X_i,
\end{equation}
where we now show the sum $T$ is a sufficient statistic for the FvML.

\begin{proposition}[Sufficiency and completeness of $T$]
\label{prop:T-complete}
The statistic $T=\sum_{i=1}^N X_i$ is sufficient and complete for $(\mu,\kappa)\in\Sph{n-1}\times(0,\infty)$.
\end{proposition}
\begin{proof}
Sufficiency is the Neyman-Fisher factorization applied to~\eqref{eq:joint-density}.
For completeness, \eqref{eq:joint-density} extends continuously to $\eta=0$ (the uniform law on $\Sph{n-1}$), giving a natural exponential family with sufficient statistic $T$ and full natural parameter space $\R^n$, which contains an open ball; the standard completeness theorem \cite[Thm.~1.6.22]{lehmann1998theory} then gives completeness of $T$ on this extended family, hence on the
original parameter set.
\end{proof}

Instead of the mean direction $\mu$ and the concentration $\kappa$ it can be useful to think of the joint density~\cref{eq:joint-density} in an alternative parameterization in terms of the mean resultant length $\bar{R}$ and sample mean direction $\hat\mu$, defined by,
\begin{equation}
    \bar{R}\coloneqq\frac{\lvert T\rvert}{N}, \quad \hat\mu\coloneqq \frac{T}{\lvert T\rvert},
\end{equation}
where for $n\ge2$ and $\kappa>0$, $T\ne0$ almost surely, so $\hat\mu$ is
well-defined almost surely.

\begin{corollary}[Sufficiency of $(\bar{R},\hat\mu)$]
    Without proof, the pair $(\bar{R},\hat\mu)$ is also sufficient.
\end{corollary}

\subsection{Existing Estimators for Concentration}
\label{sec:existing}

Note that the mean resultant length is an approximation of expected value of the norm of some FvML distributed random variable,
\begin{equation}
    \bar{R} \approx \mathbb{E}\left[\left\lVert X\right\rVert\right],
\end{equation}
which, by \cref{eq:An-to-EnormX}, means that $A_n(\hat\kappa)=\bar{R}$, and that the \emph{maximum likelihood estimator} (MLE)~\cite{fisher1953dispersion} of $\kappa$ can be written as,
\begin{equation}\label{eq:MLE}
    \hat\kappa = A_n^{-1}(\bar{R}),
\end{equation}
which is well-defined a.s.\ for $N\ge2$ since $\bar{R}<1$ a.s.\ for the continuous $n\ge2$ FvML model.
Best and Fisher \cite{bestfisher1981} showed that the bias $\mathrm{Bias}(\hat\kappa)>0$, meaning that there is always positive bias, for
$n=2,3$.
Tanabe et al.~\cite{tanabe2007parameter} derived the leading $O(N^{-1})$ bias for general $n\ge2$.

Banerjee et al.\ ~\cite{banerjee2005clustering}
proposed the estimator,
\begin{equation}
    \hat\kappa_{\mathrm{B}}\coloneqq \bar{R}(n-\bar{R}^2)/(1-\bar{R}^2),
\end{equation}
as when $N\to\infty$, $\bar{R}\to A_n(\kappa)$ a.s., meaning that 
$\hat\kappa_{\mathrm{B}}\to
A_n(\kappa)(n-A_n(\kappa)^2)/(1-A_n(\kappa)^2)$.
However, the estimator $\hat\kappa_{\mathrm{B}}$ is inconsistent.
Building on top of this, Sra~\cite{sra2012short} applied a Newton step to improve accuracy, but, as we later empirically show, increase bias.

Best and Fisher~\cite{bestfisher1981} gave corrections for $n=2,3$ which were extended by Tanabe et al.~\cite{tanabe2007parameter}  to all $n\ge2$, achieving $O(N^{-2})$ residual bias.

For large $n$ or large $\kappa$, two asymptotic approximations exist.
The high-dimensional approximation~\cite{banerjee2005clustering} uses $A_n(\kappa)\approx\kappa/n$, giving the estimator,
\begin{equation}
    \hat\kappa_{\mathrm{hd}}\coloneqq n\bar{R},
\end{equation}
while the large-$\kappa$ approximation~\cite{watson1983statistics} uses
$A_n(\kappa)\approx1-(n-1)/(2\kappa)$, to get the estimator,
\begin{equation}
    \hat\kappa_{\mathrm{lk}}\coloneqq (n-1)/(2(1-\bar{R})).
\end{equation}
Both are biased outside their respective asymptotic regimes.

For the sample size $N\ge2$, the U-statistic~\cite{hoeffding1948class}
\begin{equation}
  \label{eq:u-stat-an2}
  U_{A^2} \coloneqq  \frac{N\bar{R}^2-1}{N-1}
\end{equation}
is an unbiased estimator of $A_n(\kappa)^2$.
Indeed, by
\cref{prop:An-properties}(i),
$\mathbb{E}[X_i]=A_n(\kappa)\mu$, so for $i\ne j$,
$\mathbb{E}[X_i^\top X_j]=A_n(\kappa)^2\|\mu\|^2=A_n(\kappa)^2$.
Since $\|X_i\|=1$, expanding $\|T\|^2=N+2\sum_{i<j}X_i^\top X_j$ gives
$U_{A^2}=\binom{N}{2}^{-1}\sum_{i<j}X_i^\top X_j$.
Inverting via $A_n^{-1}(\sqrt{\max(U_{A^2},0)})$ gives a consistent but biased estimator of $\kappa$.

%\cref{sec:estimators} generalizes this unbiased statistic from
%$A_n(\kappa)^2$ to powers $A_n(\kappa)^{2\ell}$.

\section{Impossibility of estimating concentration}
%Estimable Functions of \texorpdfstring{$\kappa$}{kappa} are Analytic in \texorpdfstring{$\kappa^2$}{kappa squared}
\label{sec:impossibility}

We now prove that it is not possible to estimate the concentration $\kappa$ through the fact that all estimable functions of the concentration $\kappa$ are analytic in $\kappa^2$.

We use the joint density~\eqref{eq:joint-density} for all dimension $n\geq 1$
and sample size $N\geq 1$.
When $n=1$ we use $\Sph{0}=\{-1,+1\}$ and $\sigma$ as the counting measure.
We use the fact that the sufficient statistic $T(x)=\sum_i x_i$ has $\norm{T(x)}_2\leq N$
on $\samples$, and we write $\E_{\mu,\kappa}$ for expectation under
$P_{\mu,\kappa}^{\otimes N}$.

\begin{theorem}[Estimable functions of $\kappa$ are real-analytic in $\kappa^2$]
\label{thm:estimable-analytic-in-kappa-sq}
Let $g\colon(0,\infty)\to\R$.  If a Borel-measurable
$\hat\theta\colon\samples\to\R$ satisfies
$\E_{\mu,\kappa}[\abs{\hat\theta}]<\infty$ and
$\E_{\mu,\kappa}[\hat\theta]=g(\kappa)$ for all
$(\mu,\kappa)\in\Sph{n-1}\times(0,\infty)$, then there exists a
real-analytic function $h$ on an open neighborhood of $[0,\infty)$ in
$\R$ with $g(\kappa)=h(\kappa^2)$ for all $\kappa>0$.
\end{theorem}

\begin{proof}
\emph{Step 1: Integrability against $\sigmaN$.}
Fix any $(\mu_0,\kappa_0)$.  The pointwise lower bound
$C_n(\kappa_0)^N e^{\kappa_0\mu_0^\top T(x)}\geq C_n(\kappa_0)^N e^{-\kappa_0 N}>0$
on $\samples$ gives
\begin{equation}
  \int_{\samples}\abs{\hat\theta(x)}\,\dd\sigmaN(x)
  \leq C_n(\kappa_0)^{-N} e^{\kappa_0 N}\,\E_{\mu_0,\kappa_0}[\abs{\hat\theta}]
  <\infty,
\end{equation}
so $\hat\theta\in L^1(\sigmaN)$.

\emph{Step 2: Haar symmetrize.}
Let $dQ$ denote normalized Haar measure on the compact group $O(n)$
acting diagonally on $\samples$ by $Qx\coloneqq(Qx_1,\ldots,Qx_N)$.
The map $(Q,x)\mapsto\hat\theta(Qx)$ is jointly Borel since the diagonal
action is continuous.  Set
\begin{equation}
  E \coloneqq \Bigl\{x\in\samples : \int_{O(n)}\abs{\hat\theta(Qx)}\,dQ<\infty\Bigr\},
\end{equation}
and define
\begin{equation}
  \tilde\theta(x) \coloneqq \mathbf 1_E(x)\int_{O(n)}\hat\theta(Qx)\,dQ.
\end{equation}
By Tonelli and $O(n)$-invariance of $\sigmaN$, it is the case that,
\begin{equation}
    \iint\abs{\hat\theta(Qx)}\,\dd\sigmaN(x)\,dQ =\norm{\hat\theta}_{L^1(\sigmaN)} < \infty,
\end{equation}
meaning $\sigmaN(E^c)=0$ and $\tilde\theta\in L^1(\sigmaN)$.
For $R\in O(n)$, substituting $Q'=QR$ in the inner integral gives, 
\begin{equation}
    \int\abs{\hat\theta(QRx)}\,dQ=\int\abs{\hat\theta(Q'x)}\,dQ',
\end{equation}
so $E$ is $O(n)$-invariant under the diagonal action.
The same substitution on
$\hat\theta(QRx)$ gives $\tilde\theta(Rx)=\tilde\theta(x)$ pointwise on
$\samples$.
Fubini applied to,
\begin{equation}
    \iint\abs{\hat\theta(Qx)}\,C_n(\kappa)^N e^{\kappa\mu^\top T(x)}\,\dd\sigmaN(x)\,dQ
\leq C_n(\kappa)^N e^{\kappa N}\norm{\hat\theta}_{L^1(\sigmaN)}<\infty,
\end{equation}
together with $T(Qx)=QT(x)$ and $\mu^\top Q^\top T(y)=(Q\mu)^\top T(y)$ then gives
\begin{equation}
  \E_{\mu,\kappa}[\tilde\theta]
  = \int_{O(n)}\E_{Q\mu,\kappa}[\hat\theta]\,dQ = g(\kappa).
\end{equation}

\emph{Step 3: Real-analyticity of the natural-parameter expectation.}
For $\eta\in\R^n$ define
\begin{equation}
  \label{eq:m}
  \tilde m(\eta) \coloneqq \frac{\int\tilde\theta(x)\,e^{\eta\cdot T(x)}\,\dd\sigmaN(x)}
                                {\int e^{\eta\cdot T(x)}\,\dd\sigmaN(x)}.
\end{equation}
On any compact $K\subset\R^n$ and any multi-index $\alpha$, the bound
\begin{equation}
\abs{\tilde\theta(x)\,T(x)^{\alpha}\,e^{\eta\cdot T(x)}}
\leq N^{\abs{\alpha}}\,e^{(\sup_K\norm{\eta}_2)N}\,\abs{\tilde\theta(x)}
\end{equation}
is integrable against $\sigmaN$ since $\tilde\theta\in L^1(\sigmaN)$;
the same bound with $\tilde\theta\equiv 1$ controls the denominator.
Differentiation under the integral therefore extends both the numerator
and the denominator to real-analytic functions on $\R^n$.
The denominator equals $C_n(\norm{\eta}_2)^{-N}>0$ on $\R^n$ by
$O(n)$-invariance, so $\tilde m$ is real-analytic on $\R^n$.
For $\eta\neq 0$, with $\kappa=\norm{\eta}_2$ and $\mu=\eta/\kappa$,
$\eta\cdot T(x)=\kappa\mu^\top T(x)$ and \eqref{eq:m} reduces to
$\tilde m(\eta)=\E_{\mu,\kappa}[\tilde\theta]=g(\kappa)$.

\emph{Step 4: Rotation invariance.}
For $Q\in O(n)$ and $\eta\in\R^n$, substitute $y=Q^{-1}x$ in numerator
and denominator of \eqref{eq:m}: $O(n)$-invariance of $\sigmaN$,
$\tilde\theta(Qy)=\tilde\theta(y)$, $T(Qy)=QT(y)$, and
$(Q\eta)\cdot QT(y)=\eta\cdot T(y)$ together yield
$\tilde m(Q\eta)=\tilde m(\eta)$.

\emph{Step 5: Reduce to a line.}
Fix $v\in\Sph{n-1}$ and set $\varphi(t)\coloneqq\tilde m(tv)$,
real-analytic on $\R$.  Choose $R\in O(n)$ with $Rv=-v$
($R=-I$ for $n=1$, reflection through $v^\perp$ for $n\geq 2$).
Then $\varphi(-t)=\tilde m(R(tv))=\tilde m(tv)=\varphi(t)$, so $\varphi$
is even.

\emph{Step 6: Even real-analytic factors through $t^2$.}
On any disk $\abs{t}<r$ where $\varphi(t)=\sum_{k\geq 0}c_k t^k$, evenness
forces $c_{2\ell+1}=0$, so $\psi_0(s)\coloneqq\sum_\ell c_{2\ell}s^\ell$ is
real-analytic on $\abs{s}<r^2$ and $\varphi(t)=\psi_0(t^2)$ for $\abs{t}<r$.
Define $\psi_1(s)\coloneqq\varphi(\sqrt s)$ on $(0,\infty)$;
this is real-analytic since $\sqrt{\cdot}$ is real-analytic on $(0,\infty)$
and $\varphi$ is real-analytic on $\R$.
On the overlap $(0,r^2)$, $\psi_0$ and $\psi_1$ both equal $\varphi(\sqrt s)$,
so they coincide; gluing yields a real-analytic $\psi$ on
$(-r^2,\infty)\supset[0,\infty)$ with $\varphi(t)=\psi(t^2)$ for all $t\in\R$.

\emph{Step 7: Identify.}
For $\eta\neq 0$, rotation invariance gives
$\tilde m(\eta)=\tilde m(\norm{\eta}_2 v)=\varphi(\norm{\eta}_2)
=\psi(\norm{\eta}_2^2)=\psi(\kappa^2)$,
while Step 3 gives $\tilde m(\eta)=g(\kappa)$.
Hence $g(\kappa)=h(\kappa^2)$ with $h\coloneqq\psi$ real-analytic on a
neighborhood of $[0,\infty)$.
\end{proof}

We are now ready to prove that there does not exist an unbiased estimator for $\kappa$.

\begin{corollary}[No unbiased estimator of $\kappa$]
\label{cor:no-unbiased-kappa}
For any $n\geq 1$ and $N\geq 1$, the concentration $\kappa$ of
$\FvML(\mu,\kappa)$ admits no unbiased estimator based on $N$ i.i.d.\
observations.
\end{corollary}
\begin{proof}
If $\hat\kappa$ were unbiased, \cref{thm:estimable-analytic-in-kappa-sq}
applied with $g(\kappa)=\kappa$ would yield a real-analytic $h$ on a
neighborhood of $[0,\infty)$ with $\kappa=h(\kappa^2)$ for all $\kappa>0$,
forcing $h(s)=\sqrt{s}$ on $(0,\infty)$.
But $h'(s)=1/(2\sqrt s)\to\infty$ as $s\downarrow 0$, so $\sqrt{\cdot}$
admits no real-analytic extension across $0$, a contradiction, as required.
\end{proof}

\Cref{thm:estimable-analytic-in-kappa-sq} forces every unbiasedly estimable function of $\kappa$ to lie in the algebra of real analytic functions of $\kappa^2$, which contains $\kappa^{2\ell}$ for every $\ell\geq 1$ and their convergent combinations.
As, according to~\cref{cor:no-unbiased-kappa} there is no unbiased estimator for $\kappa$, we are forced to look at other alternatives, of which the simplest non-trivial member is $\kappa^2$ itself, corresponding to $h(s)=s$, which we prove in the next section, does have an unbiased estimator.

What this means, is that the standard parameterization of the FvML density~\cref{eq:fvml-density} is ill-suited towards probabilistic modeling, and an alternative, estimable, parameterization needs to be constructed.
Instead of concentration $\kappa$, we propose to look at,
\begin{equation}\label{eq:intensity}
    \zeta \coloneqq \kappa^2,
\end{equation}
which is the squared concentration. 
We term $\zeta$ in~\cref{eq:intensity}, as the \textit{intensity}.
The FvML density~\cref{eq:fvml-density} rewritten in terms of intesity is given by,
\begin{equation}
  \label{eq:fvml-intensity-density}
  f(x;\mu,\zeta) = C_n(\zeta)\,e^{\sqrt{\zeta}\mu^\top x},
  \quad
  C_n(\zeta) = \frac{\zeta^{n/4-1/2}}{(2\pi)^{n/2}\,I_{n/2-1}(\sqrt{\zeta})},
  \quad \zeta>0,
\end{equation}
with a simple change of variable.
Throughout the rest of this work we make use of both the concentration and the intensity as related quantities.

It is then useful to look at the squared MLE estimator $\hat\zeta_{\mathrm{MLE}}\coloneqq \hat\kappa^2$ where $\hat\kappa$ is from~\cref{eq:MLE}.
The MLE $\hat\zeta_{\mathrm{MLE}}$ is consistent and asymptotically efficient \cite{mardia2009directional} but positively biased \cite{bestfisher1981,tanabe2007parameter}, which we later argue is an undesirable property. 
%Further in this work~\cref{tab:comparators} lists standard biased alternatives to the MLE used in the experiments \cref{sec:experiments}.

%%%%%%%%%%%%%%%%%%%%%%%%%%%%%%%%%%%%%%%%%%%%%%%%%
\section{Unbiased Estimators for Inverse-Series Partial-Sum Surrogates of Intensity}
\label{sec:estimators}
%%%%%%%%%%%%%%%%%%%%%%%%%%%%%%%%%%%%%%%%%%%%%%%%%

We now provide an unbiased estimator for the intensity $\zeta$. 
This is done in thee parts: first we provide a series expansion for $\zeta$ in terms of the terms the function $A_n(\kappa)$ from~\cref{prop:An-properties}, second, we provide a way to create unbiased estimates of even powers of $A_n(\kappa)$, and  finally, we put it all together to create an unbiased estimate of the intensity $\zeta$ in terms of a power series. 
There is a small but important caveat to our claims, that our unbiased estimator would require infinitely many samples in order to truly be unbiased, but this is shown to not be an issue in numerical experiments later in this work.

% \subsection{The series expansion of \texorpdfstring{$\zeta=\kappa^2$}{zeta equals kappa squared}}
% \label{sec:kappa-sq-series}

Recall from~\eqref{eq:An-series} that $A_n$ is odd and analytic near zero.
Because $A_n'(0)>0$, the formal inverse series $A_n^{-1}$ exists near zero by the analytic implicit function theorem.
Squaring this inverse gives the intensity $\zeta$ as a series in $A_n(\kappa)^2$, which we now derive.

\begin{theorem}[Series for $\zeta$]\label{thm:kappa-sq-series}
For $n\ge2$ there exists $\rho_n\in(0,R_n]$ and coefficients $c_1,c_2,\ldots$ depending only on $n$ such that, for $0<\kappa<\rho_n$,
\begin{equation}
  \label{eq:kappa-sq-series}
  \zeta = \sum_{\ell=1}^\infty c_\ell\,A_n(\kappa)^{2\ell}.
\end{equation}
The coefficients are given by the reversion recursion: setting,
\begin{equation}
    b_k\coloneqq [A_n(\kappa)^2]_{\kappa^{2k}}=\sum_{i+j=k-1}a_ia_j,
\end{equation}
where the $a_j$ are from~\eqref{eq:An-series}, we can write the coefficients as,
\begin{equation}
  \label{eq:ck-recursion}
  c_1 = \frac{1}{b_1} = n^2,
  \qquad
  c_\ell = -\frac{1}{b_1^\ell}\sum_{j=1}^{\ell-1}c_j
    \bigl[A_n(\kappa)^{2j}\bigr]_{\kappa^{2\ell}},
  \quad \ell\ge2,
\end{equation}
where the notation $[f(\kappa)]_{\kappa^{2\ell}}$ denotes the coefficient of $\kappa^{2\ell}$
in the power series of $f$.
\end{theorem}
\begin{proof}
In order to be consistent, all derivations are performed in terms of concentration $\kappa$.
As, $A_n(\kappa)=\kappa/n+O(\kappa^3)$ that means that $A_n(\kappa)^2=\kappa^2/n^2+O(\kappa^4)$. Defining $u(w)\coloneqq A_n(\sqrt w)^2$, $u(w)$ is real-analytic on $(-R_n^2,R_n^2)$ with nonzero derivative at $w=0$.
The analytic implicit function theorem \cite[Thm.~2.2]{krantz2002primer} inverts $u$ to give $w=\sum_\ell c_\ell u^\ell$ on $\abs{u}<\tau_n$ for some $\tau_n>0$ (Lagrange-B\"urmann reversion).
Setting $\rho_n\coloneqq \sup\{\kappa\in(0,R_n) : A_n(\kappa)^2<\tau_n\}$ yields a positive interval $(0,\rho_n)\subseteq(0,R_n)$ on which $\kappa^2=\sum_\ell c_\ell A_n(\kappa)^{2\ell}$.
Expanding $\kappa^2=\sum_\ell c_\ell(\kappa^2 B(\kappa^2))^\ell$ with
$B(\kappa^2)=A_n(\kappa)^2/\kappa^2=b_1+b_2\kappa^2+\cdots$
and matching coefficients of $\kappa^{2\ell}$ gives
\eqref{eq:ck-recursion}.
\end{proof}

The first five coefficients are listed in \cref{tab:cs}.  
Each is a positive rational function of $n$ for $n\ge2$ by inspection: the numerators have nonnegative integer coefficients in $n$, and the denominators are products of positive factors $(n+2k)$.
Higher-order coefficients $c_\ell$ for $\ell\geq5$ are obtained by direct evaluation of~\eqref{eq:ck-recursion}. 
A simple closed form solution to the coefficients is of independent interest.
%The first few $a_j$ coefficients underlying~\eqref{eq:ck-recursion} are recorded in \cref{tab:as}.

\begin{table}
  \centering
  \caption{First coefficients $c_\ell$ in the series
    $\zeta=\sum_{\ell\ge1}c_\ell A_n(\kappa)^{2\ell}$,
    computed via~\eqref{eq:ck-recursion}.
    Each entry is a positive rational function of $n$ for all $n\ge2$.}
  \label{tab:cs}
  \small
  \setlength{\tabcolsep}{4pt}
  \begin{tabular}{c||llll}
    \hline
    $\ell$ & 1 & 2 & 3 & 4\\\hline
    $c_\ell$ & $n^2$ & $\dfrac{2n^3}{n+2}$ & $\dfrac{n^4(3n+20)}{(n+2)^2(n+4)}$ & $\dfrac{4n^5(n^2+14n+84)}{(n+2)^3(n+4)(n+6)}$ \\
    \hline
  \end{tabular}
\end{table}

% \subsection{Unbiased estimation of \texorpdfstring{$A_n(\kappa)^{2\ell}$}{An^{2l}}}
% \label{sec:ustat}

We now move on to the second part, deriving unbiased estimators of $A_n(\kappa)^{2\ell}$.

\begin{lemma}\label{lem:unbiased-A2l}
Let $X_1,\ldots,X_N\overset{\mathrm{iid}}{\sim}\mathrm{FvML}(\mu,\kappa)$, in terms of concentration $\kappa$, and let $i_1,\ldots,i_{2\ell}$ be distinct indices in $\{1,\ldots,N\}$.
It is then the case that,
\begin{equation}
  \label{eq:kernel-expectation}
  \mathbb{E}\left[(X_{i_1}\cdot X_{i_2})(X_{i_3}\cdot X_{i_4})
    \cdots(X_{i_{2\ell-1}}\cdot X_{i_{2\ell}})\right]
  = A_n(\kappa)^{2\ell}.
\end{equation}
\end{lemma}

\begin{proof}
First note that by independence of the $X_i$, the expectation factors as
\begin{equation}
    \mathbb{E}\left[(X_{i_1}\cdot X_{i_2})(X_{i_3}\cdot X_{i_4})
    \cdots(X_{i_{2\ell-1}}\cdot X_{i_{2\ell}})\right] = \prod_{k=1}^\ell\mathbb{E}[X_{i_{2k-1}}\cdot X_{i_{2k}}].
\end{equation}
Furthermore, for any two independent copies $X,Y\sim\mathrm{FvML}(\mu,\kappa)$, independence gives
\begin{equation}
    \mathbb{E}[X^\top Y]=\mathbb{E}[X]^\top\mathbb{E}[Y],
\end{equation}
and, from the fact that $\mu^\top\mu = 1$, and that $\mathbb{E}[X]=A_n(\kappa)\mu$, it is true that,
\begin{equation}
    \mathbb{E}[X^\top Y]=A_n(\kappa)^2.
\end{equation}
Finally taking the product over $k=1,\ldots,\ell$ gives $A_n(\kappa)^{2\ell}$, as required.
\end{proof}

We now derive unbiased estimator for arbitrary $A_n(\kappa)^{2\ell}$.

First define the sum, 
\begin{equation}
  \label{eq:Upaired}
  D_{N,\ell} \coloneqq 
  \sum_{\substack{i_1,\ldots,i_{2\ell}\\\text{distinct}}}
  (X_{i_1}\cdot X_{i_2})(X_{i_3}\cdot X_{i_4})
    \cdots(X_{i_{2\ell-1}}\cdot X_{i_{2\ell}}),
\end{equation}
which runs over all ordered $2\ell$-tuples of distinct indices, under the adjacent pairing of indices, as,
$(i_1,i_2),(i_3,i_4),\ldots,(i_{2\ell-1},i_{2\ell})$.
%The sum $D_{N,\ell}/(N)_{2\ell}$ is the ordered adjacent-pair kernel average; we use this form throughout for its direct factorization through the Gram matrix (\cref{sec:computation}).

\begin{theorem}[Unbiased estimator of $A_n(\kappa)^{2\ell}$]
\label{thm:unbiased-A2l}
For some number of samples $N\ge2\ell$, the expression,
\begin{equation}
  \label{eq:UA2l}
  \widehat{A^{2\ell}} \coloneqq \frac{D_{N,\ell}}{(N)_{2\ell}}
\end{equation}
is an unbiased estimator of $A_n(\kappa)^{2\ell}$, where $(N)_{2\ell}\coloneqq N(N-1)\cdots(N-2\ell+1)$ is the falling factorial.
\end{theorem}
\begin{proof}
Expanding the expectation of $D_{N,\ell}$, each term in the sum has distinct indices, so \cref{lem:unbiased-A2l} applies to each.
As there are $(N)_{2\ell}$ such terms, it is the case that 
\begin{equation}
    \mathbb{E}[D_{N,\ell}]=(N)_{2\ell}\cdot A_n(\kappa)^{2\ell},
\end{equation}
as required.
\end{proof}

% \subsection{The estimator \texorpdfstring{$\hat\zeta_M$}{zeta-hat-M}}
% \label{sec:estimator}

We are now ready to provide a truncated of the unbiased estimator of the intensity $\zeta$~\cref{eq:intensity}.

\begin{corollary}[Unbiased estimator of the $M$-term partial sum]
\label{thm:unbiased-kappa-sq}
For the embedded dimension $n\ge2$, a finite truncation $M\ge1$, and number of samples $N\ge2M$, the estimator,
\begin{equation}
  \label{eq:zeta-M}
  \hat\zeta_M \coloneqq  \sum_{\ell=1}^M c_\ell\,\widehat{A^{2\ell}}
  = \sum_{\ell=1}^M \frac{c_\ell\,D_{N,\ell}}{(N)_{2\ell}}
\end{equation}
is an unbiased estimator of the partial sum,
\begin{equation}
  \label{eq:partial-sum}
  \zeta_M(\kappa) \coloneqq  \sum_{\ell=1}^M c_\ell\,A_n(\kappa)^{2\ell}.
\end{equation}
As $M\to\infty$, $\zeta_M(\kappa)\to\zeta$ for every $0<\kappa<\rho_n$,
and the truncation error is
$\E[\hat\zeta_M]-\zeta = -\sum_{\ell>M}c_\ell A_n(\kappa)^{2\ell}$ on that range. 
\end{corollary}
\begin{proof}
Unbiasedness of $\hat\zeta_M$ for $\zeta_M(\kappa)$ follows by linearity
from \cref{thm:unbiased-A2l}.  Convergence $\zeta_M(\kappa)\to\zeta$ is \cref{thm:kappa-sq-series}, and the truncation formula follows from $\E[\hat\zeta_M]=\zeta_M(\kappa)$.
\end{proof}

As the infinite sum exactly converges to $\zeta$, 
we call $\hat\zeta_M$ the \emph{partial-sum U-statistic} for the intensity.
It is important to note that the estimator $\hat\zeta_M$ is only unbiased when $M\to\infty$ and that an infinite $M$ can only be achieved through infinitely many samples. In practice, we will experimentally show that this is still a substantial improvement over biased estimators for computing $\kappa$.

% Throughout the rest of the paper, $\hat\zeta_M$ denotes this
% squared-scale estimator, 
% while $\hat\kappa_M^{\mathrm{pp}}$
% (\cref{sec:positive-part}) is reserved for concentration-scale
% RMSE reporting.

\subsection{Efficient computation}
\label{sec:computation}

The most important terms of the estimator $\hat\zeta_M$ from~\cref{eq:zeta-M} to compute are the terms $D_{N,\ell}$ in~\cref{eq:Upaired}.
A naive evaluation of each would requires $O(N^{2\ell})$ dot products, which in the large-sample regime, is completely intractable.
However, because the kernel factors through pairwise inner products, every term in $D_{N,\ell}$ is a product of $\ell$ off-diagonal entries of the Gram
matrix,
\begin{equation}
    G\coloneqq XX^\top\in\R^{N\times N},\quad G_{ij}=X_i\cdot X_j.
\end{equation}
The complete adjacent-pair U-statistic $D_{N,\ell}/(N)_{2\ell}$ then admits the standard symmetric polynomial reduction to scalar contractions of $G$ \cite[Ch.~4]{mccullagh2018tensor}, thus being able to reduce the computational load.
It is possible to evaluate these contractions through liberal use of caching.
Forming $G$ costs $O(N^2 n)$ and the subsequent contraction cost is completely dominated by it, significantly reducing the computational burden.
It is the authors' experience that this version of the computation with small finite partial sum parameter $M$ works best when the dimension $n$ is large.

Another alternative to the above is to recognize the identity $G_{ij}^m=(X_i\cdot X_j)^m=X_i^{\otimes m}\cdot X_j^{\otimes m}$, from which
each contraction can be reformulated as a tensor einsum over the sample
moment tensors $t_d\coloneqq \sum_{i=1}^N X_i^{\otimes d}\in\mathbb{R}^{n^d}$,
requiring $O(Nn^M)$ preprocessing time.
This is preferred when $N\gg n^{M-1}$.

While the estimation above is theoretically sound, it is still computationally expensive.
More importantly, the partial sum statistic is limited by a lack of samples, as it has to be the case that $M \leq 2N$, thus if the number of samples is low, only a few terms of the partial sum can be computed.
We can instead look at approximating the partial sum for larger $M$ in the following fashion: for each $\ell \in \{1,\dots, M\}$, we approximate the sum in~\cref{eq:Upaired} with a finite, fixed amount of samples $P$. 
If $\ell\leq 2N$, then each sample is a term of the sum computed using a uniformly random selection of a permutation of the indices. 
Otherwise, a uniformly random selection of $2\ell$ indices (with repetition) is chosen.
In this way a biased estimate of the partial sum can be computed for large partial sum parameter $M$ even when the number of samples $N$ is small.
It is the authors' experience that this version of the computation with large finite partial sum parameter $M$ works best when the dimension $n$ is small, like $n=2$.

Computing the coefficients $c_\ell$ in~\cref{eq:ck-recursion} can also be accomplished in a computationally efficient manner through the use of the fast-Fourier transform, taking $\mathcal{O}(M\log M)$ operations.

% \subsection{Positive-part square-root estimator}
% \label{sec:positive-part}

% Because $\widehat{A^{2\ell}}$ is an average of dot-product products,
% $\hat\zeta_M$ can take negative values at finite $N$.  We therefore
% take the positive part before extracting the square root,
% \begin{equation}
%   \label{eq:kappa-cons}
%   \hat\kappa_M^{\mathrm{pp}} \coloneqq  \bigl(\hat\zeta_M\bigr)_+^{1/2},
%   \qquad (t)_+ \coloneqq  \max(t,0),
% \end{equation}
% yielding a real-valued, nonnegative square-root summary.

% Because $c_1,\ldots,c_5>0$ (by inspection of \cref{tab:cs}),
% $\zeta_M(\kappa)>0$ on $(0,\rho_n)$ at $M\le5$, so by consistency
% $\hat\zeta_M\to\zeta_M(\kappa)>0$ almost surely as $N\to\infty$;
% the positive-part truncation is therefore inactive asymptotically
% and $\hat\kappa_M^{\mathrm{pp}}\to\sqrt{\zeta_M(\kappa)}$ a.s.
% The next section measures when the finite surrogate error is smaller
% than the sampling and MLE-inversion errors it avoids.

\section{Synthetic Data Experiments}
\label{sec:synthetic}

\begin{table}
  \centering
  \small
  \begin{tabular}{ll}
    \hline
    Name & Formula or definition \\
    \hline
    MLE$^2$ & $\hat\kappa^2$, $A_n(\hat\kappa)=\bar R$ \cite{fisher1953dispersion} \\
    Banerjee & $\hat\kappa_{\mathrm{B}}^2$, $\hat\kappa_{\mathrm{B}}=\bar R(n-\bar R^2)/(1-\bar R^2)$ \cite{banerjee2005clustering} \\
    Sra & Sra-Newton refinement of $\hat\kappa_{\mathrm{B}}$ \cite{sra2012short} \\
    Tanabe BC & Tanabe $O(N^{-2})$ bias correction \cite{tanabe2007parameter} \\
    High-dim approx & $(n\bar R)^2$, high-dim approx \cite{banerjee2005clustering} \\
    Large-$\kappa$ approx & $((n-1)/(2(1-\bar R)))^2$, large-$\kappa$ approx \cite{watson1983statistics} \\
    $U_{A_n^2}$ & $(A_n^{-1}(\sqrt{\max(U_{A^2},0)}))^2$ \cite{hoeffding1948class} \\
    U-stat & This work with exact partial sum computation \\
    RU-stat & This work with random partial sum computation \\
    \hline
  \end{tabular}
    \caption{Intensity-scale estimators. All estimators depend on the data through $\bar R=\|T\|/N$ alone.}
  \label{tab:comparators}
\end{table}

% \begin{table}
%   \centering
%   \caption{Intensity-scale estimators. All estimators depend on the data through $\bar R=\|T\|/N$ alone.}
%   \label{tab:comparators}
%   \small
%   \begin{tabular}{ll}
%     \hline
%     Label & Formula or definition \\
%     \hline
%     $\hat\zeta_{\mathrm{MLE}}$ & $\hat\kappa^2$, $A_n(\hat\kappa)=\bar R$ \cite{fisher1953dispersion} \\
%     $\hat\zeta_{A^2}$ & $(A_n^{-1}(\sqrt{\max(U_{A^2},0)}))^2$ \cite{hoeffding1948class} \\
%     $\hat\zeta_{\mathrm{B}}$ & $\hat\kappa_{\mathrm{B}}^2$, $\hat\kappa_{\mathrm{B}}=\bar R(n-\bar R^2)/(1-\bar R^2)$ \cite{banerjee2005clustering} \\
%     $\hat\zeta_{\mathrm{Sra}}$ & Sra-Newton refinement of $\hat\kappa_{\mathrm{B}}$ \cite{sra2012short} \\
%     $\hat\zeta_{\mathrm{T}}$ & Tanabe $O(N^{-2})$ bias correction \cite{tanabe2007parameter} \\
%     $\hat\zeta_{\mathrm{bc}}$ & $\hat\kappa^2-1/(NA_n'(\hat\kappa))$ (delta-method) \\
%     $\hat\zeta_{\mathrm{hd}}$ & $(n\bar R)^2$, high-dim approx \cite{banerjee2005clustering} \\
%     $\hat\zeta_{\mathrm{lk}}$ & $((n-1)/(2(1-\bar R)))^2$, large-$\kappa$ approx \cite{watson1983statistics} \\
%     $\hat\zeta_{\mathrm{CS}}$ & $\hat\kappa_{\mathrm{CS}}^2$, project-derived Cox-Snell$^*$ \\
%     \hline
%     \multicolumn{2}{l}{\footnotesize $^*$Extends Best--Fisher
%       \cite{bestfisher1981} to $n\ge2$ via}\\
%     \multicolumn{2}{l}{\footnotesize $\hat\kappa_{\mathrm{CS}} \coloneqq  \hat\kappa - (1/N)\,
%       [A_n(\hat\kappa)/\hat\kappa+A_n(\hat\kappa)^3
%       -A_n(\hat\kappa)(1+(n-1)/\hat\kappa^2)]/(2A_n'(\hat\kappa))$.}\\
%     \hline
%   \end{tabular}
% \end{table}

\begin{table}[]
    \centering
 \setlength{\tabcolsep}{5pt}
  \small
  \begin{tabular}{llll}
  \hline
  Name & $N=20$ & $N=50$ & $N=100$\\
  \hline
MLE$^2$ & $0.397\pm1.014$ & $0.179\pm0.580$ & $0.081\pm0.345$ \\
Banerjee & $0.580\pm1.385$ & $0.157\pm0.555$ & $0.075\pm0.352$ \\
Sra & $0.461\pm1.085$ & $0.159\pm0.528$ & $0.069\pm0.371$ \\
Tanabe BC & $0.312\pm0.918$ & $0.138\pm0.546$ & $0.080\pm0.360$ \\
High-dim approx & $-0.053\pm0.470$ & $-0.146\pm0.305$ & $-0.180\pm0.215$ \\
Large-$\kappa$ approx & $0.084\pm0.688$ & $-0.079\pm0.324$ & $-0.137\pm0.200$ \\
$U_{A_n^2}$ & $0.127\pm0.929$ & $0.055\pm0.530$ & $0.026\pm0.355$ \\
U-stat ($M=5$) & $\mathbf{0.009\pm0.900}$ & $0.008\pm0.525$ & $-0.004\pm0.326$ \\
RU-stat ($M=50$, $B=1000$) & $0.031\pm0.938$ & $\mathbf{0.007\pm0.522}$ & $\mathbf{0.001\pm0.373}$ \\
\hline
    \end{tabular}
    \caption{Signed relative error for $n=2$, $\zeta=1$, together with one standard deviation over Monte Carlo runs.}
    \label{tab:results-n2}
\end{table}

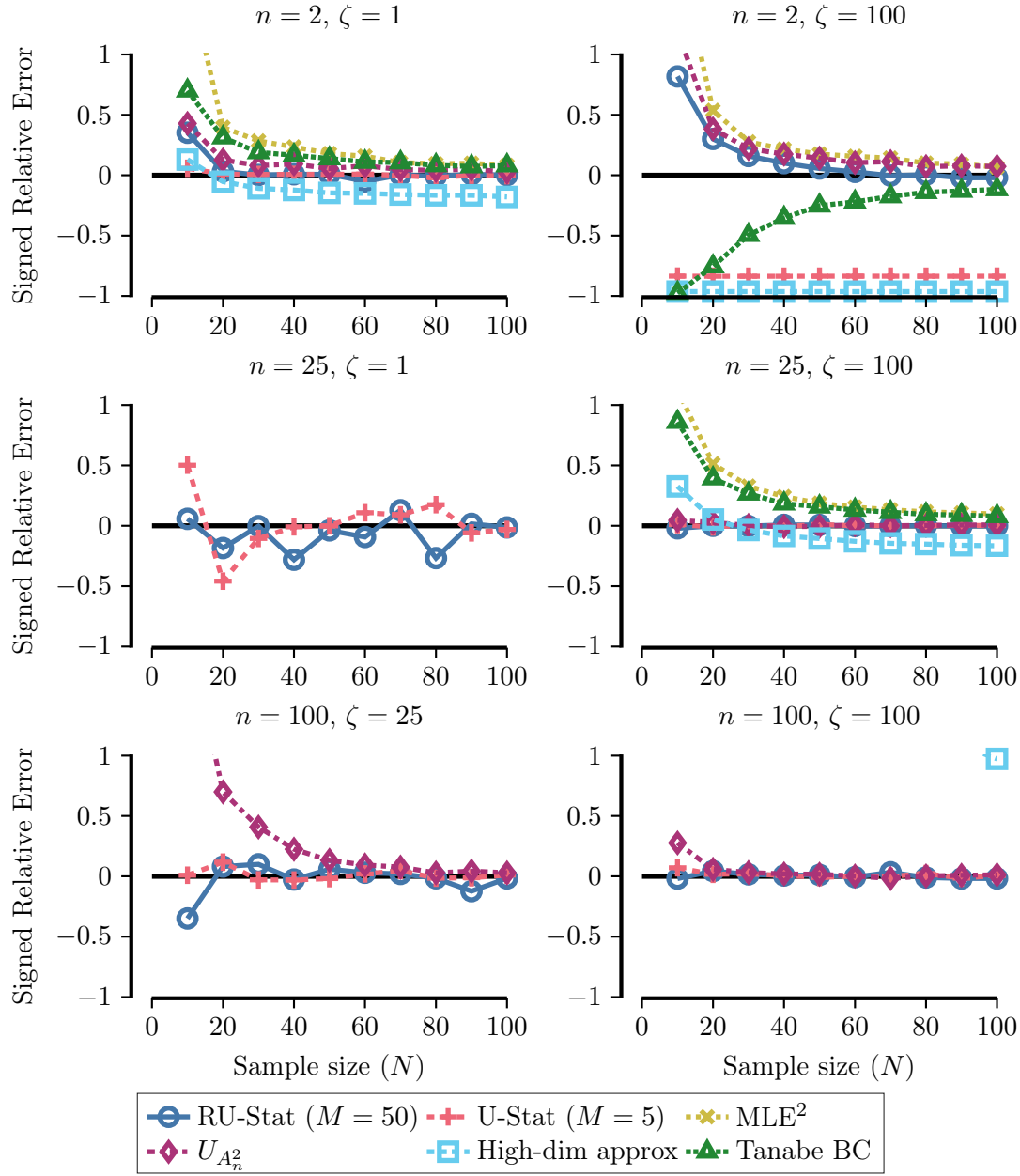
\begin{figure}
    \centering
    \begin{tikzpicture}
    \begin{axis}[clean,
        table/col sep=comma,
        xmin = -0.1,
        xmax = 100.01,
        ymin = -1.01,
        ymax = 1.01,
        clip mode=individual,
        clip = true,
        title= {$n=2$, $\zeta=1$},
        ylabel = {Signed Relative Error},
        every axis plot/.append style={line width=2pt, mark size=3.5pt}]

    \addplot [mark=none, black] coordinates {(0.0, 0.0) (100, 0)};
    
    \addplot[mark=o,color=tolblue,mark options=solid] table [x=Ns, y=incomplete_K50n2z1.00, col sep=comma] {results_aap_df.csv};

    \addplot[mark=+,color=tolred,dashed,mark options=solid] table [x=Ns, y=u_stat_K5n2z1.00, col sep=comma] {results_aap_df.csv};

    \addplot[mark=x,color=tolyellow, dotted,mark options=solid] table [x=Ns, y=mlen2z1.00, col sep=comma] {results_aap_df.csv};

    \addplot[mark=diamond,color=tolpurple, dashdotted,mark options=solid] table [x=Ns, y=u_stat_an2n2z1.00, col sep=comma] {results_aap_df.csv};

     \addplot[mark=square,color=tolcyan, densely dashed,mark options=solid] table [x=Ns, y=high_dimn2z1.00, col sep=comma] {results_aap_df.csv};

    \addplot[mark=triangle,color=tolgreen, densely dotted,mark options=solid] table [x=Ns, y=tanabe_bcn2z1.00, col sep=comma] {results_aap_df.csv};
    
    \end{axis}
    \end{tikzpicture}%
    \begin{tikzpicture}
    \begin{axis}[clean,
        table/col sep=comma,
        xmin = -0.1,
        xmax = 100.01,
        ymin = -1.01,
        ymax = 1.01,
        clip mode=individual,
        clip = true,
        title= {$n=2$, $\zeta=100$},
        every axis plot/.append style={line width=2pt, mark size=3.5pt}]

    \addplot [mark=none, black] coordinates {(0.0, 0.0) (100, 0)};
    
    \addplot[mark=o,color=tolblue,mark options=solid] table [x=Ns, y=incomplete_K50n2z100.00, col sep=comma] {results_aap_df.csv};

    \addplot[mark=+,color=tolred,dashed,mark options=solid] table [x=Ns, y=u_stat_K5n2z100.00, col sep=comma] {results_aap_df.csv};

    \addplot[mark=x,color=tolyellow, dotted,mark options=solid] table [x=Ns, y=mlen2z100.00, col sep=comma] {results_aap_df.csv};

    \addplot[mark=diamond,color=tolpurple, dashdotted,mark options=solid] table [x=Ns, y=u_stat_an2n2z100.00, col sep=comma] {results_aap_df.csv};

     \addplot[mark=square,color=tolcyan, densely dashed,mark options=solid] table [x=Ns, y=high_dimn2z100.00, col sep=comma] {results_aap_df.csv};

    \addplot[mark=triangle,color=tolgreen, densely dotted,mark options=solid] table [x=Ns, y=tanabe_bcn2z100.00, col sep=comma] {results_aap_df.csv};
    
    \end{axis}
    \end{tikzpicture}

            \begin{tikzpicture}
    \begin{axis}[clean,
        table/col sep=comma,
        xmin = -0.1,
        xmax = 100.01,
        ymin = -1.01,
        ymax = 1.01,
        clip mode=individual,
        restrict y to domain=-100:100,
        clip = true,
        title= {$n=25$, $\zeta=1$},
        ylabel = {Signed Relative Error},
        every axis plot/.append style={line width=2pt, mark size=3.5pt}]

    \addplot [mark=none, black] coordinates {(0.0, 0.0) (100, 0)};
    
    \addplot[mark=o,color=tolblue,mark options=solid] table [x=Ns, y=incomplete_K50n25z1.00, col sep=comma] {results_aap_df.csv};

    \addplot[mark=+,color=tolred,dashed,mark options=solid] table [x=Ns, y=u_stat_K5n25z1.00, col sep=comma] {results_aap_df.csv};

    \addplot[mark=x,color=tolyellow, dotted,mark options=solid] table [x=Ns, y=mlen25z1.00, col sep=comma] {results_aap_df.csv};

    \addplot[mark=diamond,color=tolpurple, dashdotted,mark options=solid] table [x=Ns, y=u_stat_an2n100z1.00, col sep=comma] {results_aap_df.csv};

     \addplot[mark=square,color=tolcyan, densely dashed] table [x=Ns, y=high_dimn25z1.00, col sep=comma] {results_aap_df.csv};

    \addplot[mark=triangle,color=tolgreen, densely dotted,mark options=solid] table [x=Ns, y=tanabe_bcn25z1.00, col sep=comma] {results_aap_df.csv};
    
    \end{axis}
    \end{tikzpicture}%
    \begin{tikzpicture}
    \begin{axis}[clean,
        table/col sep=comma,
        xmin = -0.1,
        xmax = 100.01,
        ymin = -1.01,
        ymax = 1.01,
        clip mode=individual,
        clip = true,
        title= {$n=25$, $\zeta=100$},
        every axis plot/.append style={line width=2pt, mark size=3.5pt}]

    \addplot [mark=none, black] coordinates {(0.0, 0.0) (100, 0)};
    
    \addplot[mark=o,color=tolblue,mark options=solid] table [x=Ns, y=incomplete_K50n25z100.00, col sep=comma] {results_aap_df.csv};

    \addplot[mark=+,color=tolred,dashed,mark options=solid] table [x=Ns, y=u_stat_K5n25z100.00, col sep=comma] {results_aap_df.csv};

    \addplot[mark=x,color=tolyellow, dotted,mark options=solid] table [x=Ns, y=mlen25z100.00, col sep=comma] {results_aap_df.csv};

    \addplot[mark=diamond,color=tolpurple, dashdotted,mark options=solid] table [x=Ns, y=u_stat_an2n25z100.00, col sep=comma] {results_aap_df.csv};

     \addplot[mark=square,color=tolcyan, densely dashed,mark options=solid] table [x=Ns, y=high_dimn25z100.00, col sep=comma] {results_aap_df.csv};

    \addplot[mark=triangle,color=tolgreen, densely dotted,mark options=solid] table [x=Ns, y=tanabe_bcn25z100.00, col sep=comma] {results_aap_df.csv};
    
    \end{axis}
    \end{tikzpicture}

        \begin{tikzpicture}
    \begin{axis}[clean,
        table/col sep=comma,
        xmin = -0.1,
        xmax = 100.01,
        ymin = -1.01,
        ymax = 1.01,
        clip mode=individual,
        restrict y to domain=-100:100,
        clip = true,
        title= {$n=100$, $\zeta=25$},
        xlabel = {Sample size ($N$)},
        ylabel = {Signed Relative Error},
        every axis plot/.append style={line width=2pt, mark size=3.5pt}]

    \addplot [mark=none, black] coordinates {(0.0, 0.0) (100, 0)};
    
    \addplot[mark=o,color=tolblue,mark options=solid] table [x=Ns, y=incomplete_K50n100z25.00, col sep=comma] {results_aap_df.csv};

    \addplot[mark=+,color=tolred,dashed,mark options=solid] table [x=Ns, y=u_stat_K5n100z25.00, col sep=comma] {results_aap_df.csv};

    \addplot[mark=x,color=tolyellow, dotted,mark options=solid] table [x=Ns, y=mlen100z25.00, col sep=comma] {results_aap_df.csv};

    \addplot[mark=diamond,color=tolpurple, dashdotted,mark options=solid] table [x=Ns, y=u_stat_an2n100z25.00, col sep=comma] {results_aap_df.csv};

     \addplot[mark=square,color=tolcyan, densely dashed] table [x=Ns, y=high_dimn100z25.00, col sep=comma] {results_aap_df.csv};

    \addplot[mark=triangle,color=tolgreen, densely dotted,mark options=solid] table [x=Ns, y=tanabe_bcn100z25.00, col sep=comma] {results_aap_df.csv};
    
    \end{axis}
    \end{tikzpicture}%
    \begin{tikzpicture}
    \begin{axis}[clean,
        table/col sep=comma,
        xmin = -0.1,
        xmax = 100.01,
        ymin = -1.01,
        ymax = 1.01,
        clip mode=individual,
        clip = true,
        title= {$n=100$, $\zeta=100$},
        xlabel = {Sample size ($N$)},
        every axis plot/.append style={line width=2pt, mark size=3.5pt}]

    \addplot [mark=none, black] coordinates {(0.0, 0.0) (100, 0)};
    
    \addplot[mark=o,color=tolblue,mark options=solid] table [x=Ns, y=incomplete_K50n100z100.00, col sep=comma] {results_aap_df.csv};

    \addplot[mark=+,color=tolred,dashed,mark options=solid] table [x=Ns, y=u_stat_K5n100z100.00, col sep=comma] {results_aap_df.csv};

    \addplot[mark=x,color=tolyellow, dotted,mark options=solid] table [x=Ns, y=mlen100z100.00, col sep=comma] {results_aap_df.csv};

    \addplot[mark=diamond,color=tolpurple, dashdotted,mark options=solid] table [x=Ns, y=u_stat_an2n100z100.00, col sep=comma] {results_aap_df.csv};

     \addplot[mark=square,color=tolcyan, densely dashed,mark options=solid] table [x=Ns, y=high_dimn100z100.00, col sep=comma] {results_aap_df.csv};

    \addplot[mark=triangle,color=tolgreen, densely dotted,mark options=solid] table [x=Ns, y=tanabe_bcn100z100.00, col sep=comma] {results_aap_df.csv};
    
    \end{axis}
    \end{tikzpicture}

\begin{tikzpicture} 
    \begin{axis}[%
    hide axis,
    xmin=10,
    xmax=50,
    ymin=0,
    ymax=0.4,
    every axis plot/.append style={line width=2pt, mark size=3.5pt},
    legend style={draw=white!15!black,legend cell align=left, legend columns=3}
    ]
    \addlegendimage{tolblue,mark=o,mark options=solid}
    \addlegendentry{RU-Stat ($M=50$)};
    \addlegendimage{tolred,dashed,mark=+,mark options=solid}
    \addlegendentry{U-Stat ($M=5$)};
    \addlegendimage{tolyellow,dotted,mark=x,mark options=solid}
    \addlegendentry{MLE$^2$};
    \addlegendimage{tolpurple,dashdotted,mark=diamond,mark options=solid}
    \addlegendentry{$U_{A_n^2}$};
    \addlegendimage{tolcyan,densely dashed,mark=square,mark options=solid}
    \addlegendentry{High-dim approx};
    \addlegendimage{tolgreen,densely dotted,mark=triangle,mark options=solid}
    \addlegendentry{Tanabe BC};
    \end{axis}
\end{tikzpicture}
\caption{Comparison of the signed relative error for a selection of various intensity estimators over several dimension parameters $n$, intensities $\zeta$, and sample sizes $N$.}
\label{fig:synth-exp}
\end{figure}

We first look at estimating the intensity through completely synthetic experiments, where all the data are derived from synthetic simulation of the FvML distribution through a computational method analogous to that presented in~\cite{pinzon2023fast}.
As we have proven that unbiased estimation of concentration $\kappa$ is impossible, we look at a wide range of other estimators derived for the concentration $\kappa$, and square them to derive approximate estimators for the intensity $\zeta$. A table of such estimators is found in~\cref{tab:comparators}.

We perform a grid study on several choices of dimension, $n\in\{2, 25, 100\}$, several choices of intensity, $\zeta\in\{1, 25, 100\}$ and over sample sizes ranging from $N=10$ to $N=100$ with an interval of 10 in between.
For each experiment, $N$ iid samples are drawn from the FvML distribution with intensity $\zeta$, and the estimator computed.
The signed relative error is computed for each estimator, and the results averaged over $1000$ Monte Carlo runs.
For the U-statistic with exact partial sum computation, the partial sum is limited to $M=5$ as the partial sum is only computable up to order $5$ for $N=10$ samples.
For the U-statistic computed with random samples, $B=1000$ samples are taken either from permuted indices or from random indices, and the partial sum is limited to $M=50$, as that is the practical limit of numerical stability for some of the coefficients for higher dimensions. 
As we wish to both show the performance of the methods and their relative bias, we choose to report the signed relative error of all the estimators,
\begin{equation}
    \operatorname{SRE}(\hat\zeta,\zeta) = \frac{\hat\zeta - \zeta}{\zeta},
\end{equation}
where $\hat\zeta$ is an estimator of the intensity and $\zeta$ is the true intensity.
For sampling from the FvML distribution we use the method from~\cite{pinzon2023fast}.
A selection for some of the estimators for six of nine experiments can be seen in~\cref{fig:synth-exp}.

We first look at the most trivial nominal case of dimension $n=2$ and intensity $\zeta=1$. 
We report the results in~\cref{tab:results-n2}.
For $N=20$, the MLE has a mean bias of positive $40\%$, meaning that it significantly overestimates the true intensity.
Several other estimators fail significantly for this case.
The high dimensional approximation is the best completing estimator with a signed error of $-5\%$.
The $U_{A_n^2}$ estimator also performs well, as it is similar to our partial sum estimator with truncation $M=1$.
The best estimator for this sample size is the U-statistic estimator with $M=5$, with a signed error of less than $1\%$.

For the case $n=2$ and $\zeta=100$, the random U-statistic estimator with $M=50$ performs the best for small sample sizes, with many estimators achieving good asymptotic convergence.
The U-statistic with $M=5$ performs poorly, which signals that in this regime, significantly more than $M=5$ terms are required for convergence of the estimators.

For the case $n=25$ and $\zeta=1$, the only two estimators that have signed relative error in the range $[-1,1]$ are our U-statistic estimators, with even the $U_{A_n^2}$ estimator completely failing. This signals that in the high-dimensional, low intensity regime, the U-statistic estimators proposed herein are the only real choice for unbiased intensity estimation.

For the case $n=25$, and $\zeta=100$, three estimators, all based on U-statistics, all have virtually no relative error for all sample sizes. 
The partial sum U-statistic with $M=5$, surprisingly performs really well in this case, likely because the series converges faster as the dimension $n$ increases.
The MLE estimator, the Tanabe BC estimator and the high-dimensional approximation all perform decently well, but fail to produce unbiased results even for relatively high sample sizes, as is expected.

For the case $n=100$ and $\zeta=1$, all estimators performed poorly for the chosen sample sizes.
It is likely because the variance in the mean direction estimate depends highly on the concentration/intensity, with higher values leading to lower variance, thus leading to a more accurate estimator.

We therefore look at the case $n=100$ and $\zeta=25$. 
Again, only the U-statistic estimator perform well here, with both of the estimators proposed in this work performing very well in the undersampled case of $N\ll n$, which was the original motivation for this work.

Finally we look at the case $n=100$, and $\zeta=100$, with both high dimension and high intensity.
Again, only the U-statistic estimators perform well, with all three virtually indistinguishable from each other, except in the case of smallest sample size $N=10$, where the $U_{A_n^2}$ estimator performs the worst. 
Of note is the fact that high dimension approximation estimator makes a brief appearance for $N=100$, meaning that for high dimension, high intensity and high sample size, it indeed might be a good estimator, but only asymptotically so.

The only estimator that had consistent and approximately unbiased results over all attempted scenarios is that of the randomized U-statistic with $M=50$.
As estimating the concentration parameter $\kappa$ in an unbiased way was proven to be impossible in this work, this makes this intensity estimator the only viable estimator over all possible unknown scenarios in known to the authors.

We have additionally ran experiment for smaller $\zeta$, such as $\zeta = 0.01$, but the results were extremely unstable for all estimators, even for $n=2$.
It is likely the fact that estimating extremely small intensity requires large sample sizes for these estimators, or that estimators for extremely small intensity $\zeta$ have to be specially constructed.

%%%%%%%%%%%%%%%%%%%%%%%%%%%%%%%%%%%%%%%%%%%%%%%%%%%%%%%%
\section{New York Taxi Data}
\label{sec:taxi}
%%%%%%%%%%%%%%%%%%%%%%%%%%%%%%%%%%%%%%%%%%%%%%%%%%%%%%%%

\begin{table}[]
    \centering
    \setlength{\tabcolsep}{5pt}
  \small
  \begin{tabular}{lrrrr}
    \hline
    Zone Name & $N$ & $\hat\zeta_{50}$ & $\hat\zeta_{\mathrm{MLE}}$ & $\Delta$ \\
    \hline
Springfield Gardens South & 22 & 388.04 & 592.50 & +204.46\\
Baisley Park & 35 & 371.42 & 560.03 & +188.61\\
Bensonhurst East & 22 & 256.66 & 322.89 & +66.23\\
Dyker Heights & 23 & 215.42 & 251.82 & +36.40\\
South Jamaica & 25 & 15.24 & 34.28 & +19.03\\
Saint Michaels Cemetery/Woodside & 106 & 31.80 & 47.91 & +16.11\\
Westchester Village/Unionport & 53 & 8.07 & 16.63 & +8.55\\
Washington Heights North & 76 & 11.62 & 16.03 & +4.41\\
Brighton Beach & 11 & 13.73 & 17.30 & +3.57\\
Midwood & 10 & 8.00 & 11.20 & +3.19\\
\hline
\multicolumn{5}{l}{\footnotesize $\Delta = \hat\zeta_{\mathrm{MLE}}-\hat\zeta_{50}$.
    Zones with $N < 10$ are not shown.} \\
\hline
    \end{tabular}
    \caption{Ten of the highest intensity discrepancy taxi zones for the 12th of January 2012.}
    \label{tab:selected-zones}
\end{table}

% \pgfplotsset{colormap={mymap}{
% rgb=(0.99532488, 0.8947328, 0.75266436)
% rgb=(0.99215686, 0.75371011, 0.53883891)
% rgb=(0.96824298, 0.49142637, 0.32287582)
% rgb=(0.83981546, 0.18380623, 0.11870819)
% }}

\pgfplotsset{colormap={mymap}{
rgb=(0.99215686, 0.75371011, 0.53883891)
rgb=(0.99215686, 0.75371011, 0.53883891)
}}

% \pgfplotsset{colormap={mymap}{
%   rgb=(0.8,0.0,0.0)
%   rgb=(0.0,0.8,0.0)
%   rgb=(0.0,0.0,0.8)
% }}

\definecolor{internal1}{rgb}{0.83981546, 0.18380623, 0.11870819}

\begin{figure}
    \centering
    \begin{tikzpicture}
    \begin{axis}[clean,
        table/col sep=comma,
        xmin = 5,
        xmax = 30000.01,
        ymin = 0.0005,
        ymax = 1100.01,
        x = 1cm,
        y = 0.5cm,
        xmode=log,
        ymode=log,
        y filter/.code={\pgfmathparse{abs(#1)}}
        clip mode=individual,
        clip = true,
        domain=5:30000,
        samples=101,
        smooth,
        ylabel = {$\lvert \hat\zeta_{\text{MLE}}-\hat\zeta_{50}\rvert$},
        xlabel={Trip count $N$},
        every axis plot/.append style={line width=2pt, mark size=2.75pt}]

    \addplot[mark=none, dashed] {10^(-0.2738*log10(x)-0.0274)};

    \addplot[line width=1.25pt,black,only marks, mark options={fill=internal1}] table [x=Ns, y=zdiff, col sep=comma] {taxi_z.csv};

    \end{axis}
    \end{tikzpicture}
\caption{A representation of the MLE versus U-statistic discrepancy. The $x$-axis represents the zone trip count, with the $y$-axis being the discrepancy. Each point is data for one zone. The dashed line represents the line of best fit in logarithmic space.}\label{fig:taxi-sample-diff}
\end{figure}
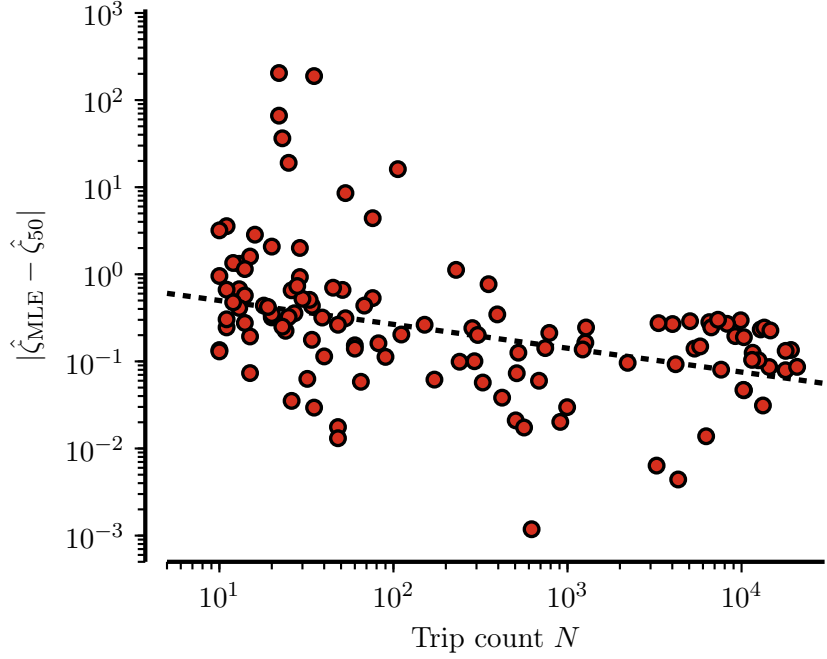

\begin{figure}
  \centering
  \includegraphics[width=0.95\linewidth]{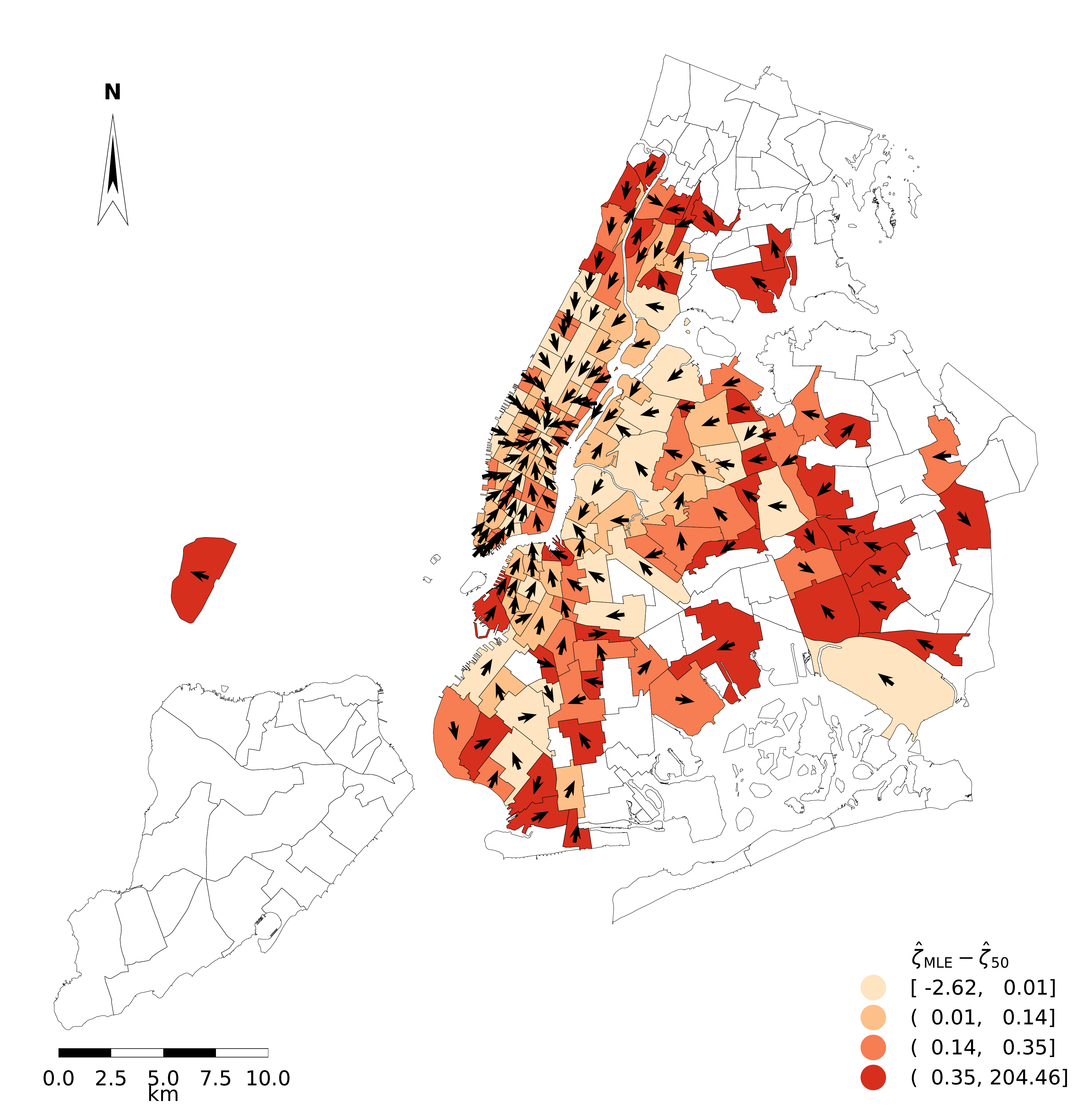}
  \caption{NYC taxi pickup zones colored by the discrepancy
    $\hat\zeta_{\mathrm{MLE}}-\hat\zeta_{50}$. Arrows indicate
    the outbound mean direction in each zone. The discrepancy between the MLE estimator and the partial sum U-statistic estimator is divided over four quantiles.}
  \label{fig:taxi-layout}
\end{figure}

New York City taxi (NYC) trip records are obtained from the New York City Taxi \& Limousine Commission (NYC TLC)\cite{nyctlc2012}. Detailed taxi trip records, including pick-up and drop-off coordinates, are available for each trip as a NYC city-level open data policy. However, after summer 2015, NYC TLC stopped sharing exact pick-up and drop-off coordinates to protect passengers' privacy. Instead, NYC TLC introduced taxi zones, which are developed based on neighborhood cultures in NYC. There are 263 taxi zones covering the five boroughs (Brooklyn, Bronx, Manhattan, Queens, and Staten Island) plus Newark Airport in Newark, New Jersey. 

In this study, we apply our estimator to taxi trip records for January 12th, 2012, a day of no significance. This date is randomly select based on two conditions. First, it is before the policy stopped exact location sharing in 2015. And second, there was no events occurring in NYC that significantly changed people's travel behavior \cite{jiang2024sensor}. We first filter out trip records that did not fall inside NYC boundary. After this step, a total of $478{,}884$ trips across $234$ taxi zones remain. In order to make a valid estimation, we further remove taxi zones with less than 10 trips from this analysis. This results in $159$ zones are used in this study. For each trip, we first identify the taxi zone where the trip starts using point-in-polygon method. Then, we replace the starting point with the centroid $(o_x,o_y)$ of the corresponding taxi-zone. We keep the exact trip destination in the original coordinates $(d_x,d_y)$ to ensure accurate directional estimations. We use EPSG:6538 as the projected coordinate system, a New York State projection for New York Long Island area, in the unit of meters. 

For each trip we then form the displacement vector,
\begin{equation}
    v_i = (d_{x,i} - o_x^{(z)},\; d_{y,i} - o_y^{(z)}) \in \R^2,
\end{equation}
and subsequently form dimensionless directional unit vectors,
\begin{equation}
    v_i \leftarrow \frac{v_i}{\norm{v_i}}.
\end{equation}
where the resulting unit vectors are points on the unit circle, and are treated as our data.
We then assume that the data are FvML distributed with $n=2$, which coincides with the von~Mises
distribution.
We discard all zones with less than $N=10$ trip counts, thus the zone trip counts range from $10$ to $21{,}275$, with a median of $281$.

Given the results in~\cref{sec:synthetic}, we make use of the randomized approach for computing the intensity U-statistic with $M=50$.
For each zone, we take all trips originating in that zone and compute $\hat\zeta_{50}$ and the MLE $\hat\zeta_{\mathrm{MLE}}$.
\Cref{tab:selected-zones} reports the trip counts, computed quantities, and their differences for a selection of ten highest discrepancy zones.
As can be seen, the largest discrepancies occur occur mostly in small trip count zones, with relatively high computed intensities.

\Cref{fig:taxi-sample-diff} is a visualization of the Trip counts $N$ versus discrepancies over all the zones. 
As can be seen, there is a slight negative correlation between trip count and discrepancy.
This is to be expected, as both methods should converge to each other in the limit of sample size.

A map \cref{fig:taxi-layout} maps the per-zone discrepancy $\hat\zeta_{\mathrm{MLE}}-\hat\zeta_{50}$ across the city along with each zone's mean outbound direction. 
As shown in \cref{fig:taxi-layout}, the dominant travel direction from taxi zones in Manhattan point points toward the Midtown Manhattan, which is the city's central business district. 
Important landmarks, such as the Empire State Building, Time Square, the headquarters of the United Nations and Rockefeller Center. 
In addition, Grand Central Terminal and Penn Station, the two important regional and local transportation hubs, are also located here. 
This pattern is consistent with previous studies that identified Midtown as one of the most active location for taxi activities \cite{jiang2021novel,jiang2025comparative}. 
The colors of \cref{fig:taxi-layout} show the discrepancies for each zone. 
The low discrepancies in high-volume zones, such as JFK airport, demonstrate that this method is consistent with intensity calculations with large samples sizes in taxi zones with large trip destination variances \cite{jiang2024entropy}.
The high discrepancies in low-volume zones, indicate that the MLE is likely poorly adapted to this use case.

\section{Descriptive Concentration in JoSE Embeddings}
\label{sec:jose}

\begin{table}[htbp]
  \centering
  \setlength{\tabcolsep}{5pt}
  \small
  \begin{tabular}{rrrrrl}
    \hline
    Cluster & $N$ & $\hat\zeta_3$ & $\hat\zeta_{\mathrm{MLE}}$ &
      $\Delta$ & Top words \\
    \hline
    14 &  446 & 9{,}256 & 11{,}442 & $+2{,}186$ & mezalimi, birinci, osmanli, harb \\
    13 & 3300 & 7{,}170 &  8{,}208 & $+1{,}038$ & 12, 11, ), 27, 19 (date tokens) \\
    11 & 1170 & 5{,}969 &  6{,}607 & $+638$ & fernandez, shortstop, smoltz, maddux \\
    17 & 2095 & 5{,}771 &  6{,}343 & $+572$ & c4u, 'l, 7u, 3o, w8 (encoding artefacts) \\
     4 & 2304 & 5{,}451 &  5{,}937 & $+486$ & 887, 1216, 1005, 881, cgy (hockey scores) \\
     7 & 1769 & 5{,}050 &  5{,}446 & $+396$ & 250mb, adaptors, opti, ne2000 \\
     2 & 1390 & 4{,}997 &  5{,}385 & $+388$ & sclerosis, estrogen, hormone \\
     1 & 1860 & 4{,}649 &  4{,}962 & $+313$ & pc, 3d, mac, digital, data \\
     9 & 1630 & 4{,}344 &  4{,}605 & $+261$ & auroral, spacecraft, orbiting, solar \\
    15 & 1609 & 4{,}145 &  4{,}374 & $+229$ & covenants, ecclesiastes, hebrews \\
     10 & 2004 & 4{,}067 &  4{,}281 & $+214$ & xlookupstring, gzip, xtappaddinput \\
    16 & 1600 & 3{,}966 &  4{,}168 & $+202$ & ascot, romeo, corvettes, beemer \\
    12 & 2335 & 3{,}933 &  4{,}127 & $+194$ & blinds, seam, poking, woodstove \\
    18 & 2628 & 3{,}839 &  4{,}019 & $+180$ & husc8, wraith, kmagnacca, uceng \\
    19 & 2919 & 3{,}756 &  3{,}924 & $+168$ & reactionary, seeming, enjoyable \\
     3 & 2759 & 3{,}739 &  3{,}905 & $+166$ & com, netcom, ileaf, nynexst, ssd \\
     6 & 3078 & 3{,}726 &  3{,}890 & $+164$ & homosexuality, issue, issues, case \\
    20 & 3455 & 3{,}643 &  3{,}796 & $+153$ & mserv, wkuvx1, anu, hogg, nermal \\
     8 & 1630 & 3{,}415 &  3{,}548 & $+133$ & khojaly, cypriot, unrest, cemetery \\
     5 & 2430 & 3{,}193 &  3{,}300 & $+107$ & disclose, sponsorship, congressmen \\
    \hline
\multicolumn{5}{l}{\footnotesize $\Delta = \hat\zeta_{\mathrm{MLE}}-\hat\zeta_{3}$.} \\
\hline
  \end{tabular}
  \caption{Partial-sum U-statistic $\hat\zeta_3$ and squared MLE reference
$\hat\zeta_{\mathrm{MLE}}$ and their discrepancy for all $20$ spherical $K$-means clusters of the JoSE/20\,Newsgroups vocabulary, sorted by $\hat\zeta_3$.  The sample size $N$ is cluster size, and  top words are nearest to the cluster mean direction by cosine similarity.}
  \label{tab:jose-kappa}
\end{table}

We move on to our final example, the Joint Spherical Embeddings(JoSE) method of Meng et al.~\cite{meng2019spherical}.
JoSE is a Word2Vec-style model that constrains word vectors to lie on the unit sphere and trains them under a margin-based cosine-similarity objective, so the learned representations live on $\Sph{n-1}$, for some arbitrary user-defined $n$.
Given the results in~\cref{sec:synthetic}, we make use of the deterministic approach for computing the intensity U-statistic with $M=3$.
Thus for this example we make use of the finite surrogate $\hat\zeta_3$ computed using the direct approach with a user chosen $n=100$.

The corpus we make use of is 20 Newsgroups~\cite{lang1995newsweeder} ($18\,846$ posts, 20 categories), with a $42\,411$-word vocabulary after lowercasing, stop-word removal, and a minimum count of 5.
Each word is represented as a 100-dimensional unit vector on $\Sph{99}$; the training for the 100-dimensional JoSE embeddings distributed with the original repository is used.

To obtain descriptive groups of vectors on the sphere, we apply spherical $K$-means~\cite{dhillon2001concept} to the vocabulary with $K=20$ clusters.
Cluster membership is assigned by the nearest centroid under cosine similarity.
Cluster sizes range from $446$ to $3\,455$ words (mean $2\,121$), well above the $N\ge6$ minimum required by $\hat\zeta_3$.
For each cluster we report $\hat\zeta_3$ alongside the intensity MLE $\hat\zeta_{\text{MLE}}$.

\Cref{tab:jose-kappa} reports $\hat\zeta_3$ and
$\hat\zeta_{\mathrm{MLE}}$ for each of the $20$ clusters, and the discrepancy between the two.  
The squared MLE is larger than $\hat\zeta_3$ in every cluster, with relative excess over $\hat\zeta_3$ ranging from $1.7\%$ (cluster~4) to $23.6\%$ (cluster~13).
Similar to the taxi data example, the gap is largest for the smallest, most tightly concentrated clusters.

Clusters with a high value of $\hat\zeta_3$ are descriptively associated with narrow-context vocabulary (Turkish tokens; baseball player names).
Clusters with a low value of $\hat\zeta_3$ contain words that appear across many contexts (general political and discourse vocabulary).

As the synthetic data experiments suggested, for a high dimension $n$ and high intensity $\zeta$, the U-statistic estimator very likely is computing close to the correct intensity, while the MLE is not.
When clustering is used to build density estimates of the data using the FvML distribution, the partial sum U-statistic estimator is the only consistent and unbiased choice.

\section{Conclusion}
\label{sec:conclusion}

In this work we have established two complementary results about the Fisher-von~Mises-Langevin distribution.
First, that no unbiased estimator of the concentration parameter $\kappa$ exists for any $n\ge1$, $N\ge1$.
Second, that, what we call the intensity, namely the squared concentration $\zeta=\kappa^2$ does admit an estimator that is unbiased for an arbitrarily accurate polynomial approximation.
Its construction combines coefficient reversion of the $A_n$ power series with U-statistics for products of inner products of FvML observations.
We provide a way to approximate this U-statistic for arbitrary sums using a randomized approach, and empirically show that this estimator is superior to all others tested for most practical cases.

Several open questions remain:
First, does there exist an estimator $\hat\zeta$ that is unbiased for the full infinite series?
The coefficient reversion approach generalizes, meaning $\kappa^{2\ell}$ has a convergent series in powers of $A_n(\kappa)^2$ for arbitrary $\ell$, but does the same U-statistic
construction apply?
Does this estimator achieve the Cram\'er--Rao bound?

Future work will be concerned with extending the results for other distributions such as the Watson distribution~\cite{sra2013multivariate}, to spherical kernel density estimation~\cite{ley2017modern, pewsey2021recent} and to pose estimation~\cite{zanetti2009multiplicative, zanetti2018fully,lu2026pose}.

\section*{Acknowledgments}
Large language models, Perplexity.ai and Anthropic Claude were used for improving language and clarity during the editing process.
The authors assume responsibility for all content.

\bibliographystyle{siamplain}
\bibliography{refaap}

\end{document}